\documentclass[11pt]{amsart}

\usepackage[dvipsnames]{xcolor}
\usepackage{enumitem}
\usepackage{amssymb,graphicx,mathtools}
\usepackage{float}
\usepackage[margin=1in]{geometry}
\usepackage{etoolbox}

\usepackage[square,longnamesfirst]{natbib}

\theoremstyle{plain}
\newtheorem{theorem}{Theorem}[section]
\newtheorem{proposition}[theorem]{Proposition}

\newtheorem{lemma}[theorem]{Lemma}
\theoremstyle{definition}
\newtheorem{problem}{Problem}

\usepackage[
citecolor=PineGreen,
colorlinks=true,
linkcolor=RoyalBlue,
urlcolor=BrickRed
]{hyperref}

\newcommand{\Z}{\mathbb Z}
\newcommand{\R}{\mathbb R}
\renewcommand{\P}{\mathbb P}
\newcommand{\E}{\mathbb E}
\newcommand{\one}{\mathbf 1}
\newcommand{\dist}{\operatorname{dist}}
\newcommand{\Reff}{\mathcal R}
\newcommand{\defn}[1]{\textbf{#1}}

\makeatletter
\patchcmd{\@settitle}{\uppercasenonmath\@title}{\Large}{}{}
\patchcmd{\@setauthors}{\MakeUppercase}{\large}{}{}
\makeatother

\title[ORRW has range exponent $\geq d/(d+1)$]{Once-reinforced random walk on $\Z^d$ \\ has range exponent at least $d/(d+1)$}
\author{Ahmed Bou-Rabee}
\author{Yuval Peres}

\begin{document}

	\begin{abstract}
		In once-reinforced random walk, introduced by \citet{Davis1990}, every edge starts with weight one and is assigned weight $\beta\geq1$ after its first crossing.  At each step the walker chooses an incident edge with probability proportional to its weight. For sufficiently large reinforcement $\beta$, the expected range in the first $n$ steps was predicted by \cite{OrdemannBerkolaikoHavlinBunde2000} and independently by \cite{Beffara2011} to have order $n^{\frac{d}{d+1}}$ on $\Z^d$.  We prove the corresponding lower bound in every dimension $d\geq2$ and for every reinforcement parameter $\beta$.
	\end{abstract}

	\maketitle
	\begin{figure}[!b]
		\centering
		\includegraphics[width=\textwidth]{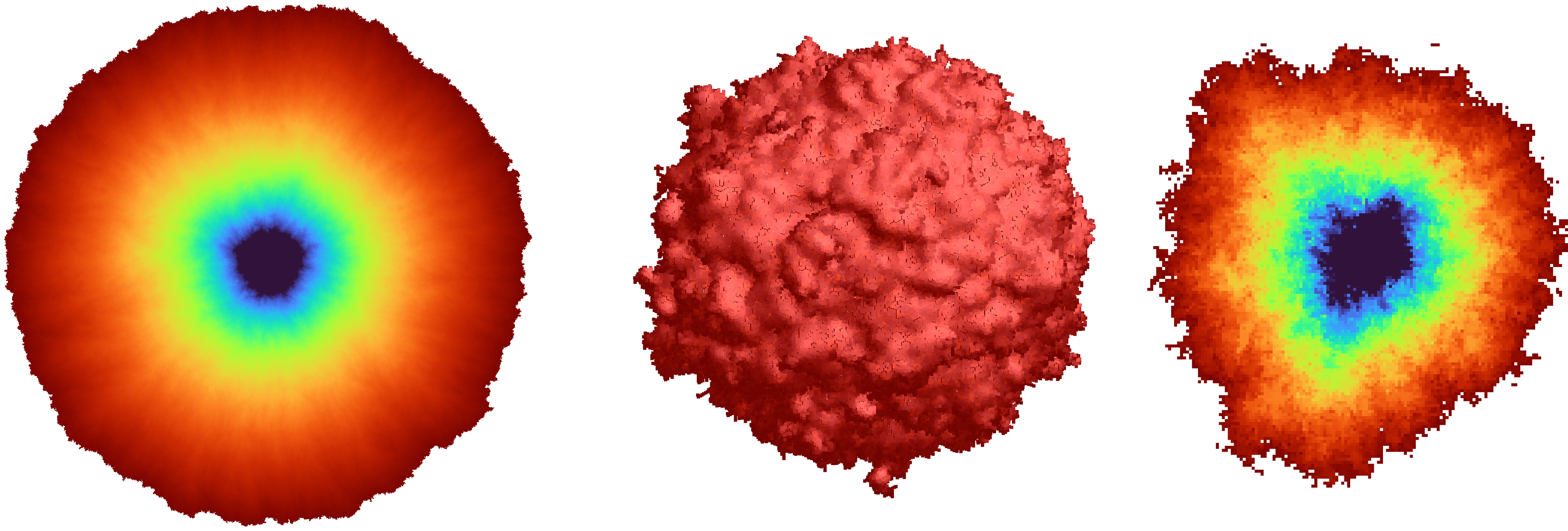}
		\caption{The range of once-reinforced random walk with $\beta=50$: on
		$\Z^2$ after $3 \times10^{10}$ steps, left; on $\Z^3$ after
		$10^9$ steps, drawn as a solid, middle; and a slice
		through the origin, right.  The color ranges on a logarithmic scale from dark blue at the sites visited first, through green and orange, to dark red at those visited last.}
		\label{fig:range}
	\end{figure}

	\section{Introduction}\label{sec:introduction}

	In \defn{once-reinforced random walk} on $\Z^d$, every edge has weight one
	until the walk first crosses it and weight $\beta\geq1$ from then on. At
	each step the walker moves to one of the $2d$ neighbors  with probability proportional to the weight of the edge it crosses.  For $\beta>1$, the walk favors previously crossed edges. At $\beta=1$, it is simple random walk.

	\citet{Davis1990} introduced the model on $\Z$ and proved that it returns to the origin infinitely often.  In higher dimensions very little is known.  For $2\leq d\leq5$ no value of
	$\beta>1$ is known to give either recurrence or transience; and for $d\geq6$
	transience is known only when $\beta$ is close to one
	\citep[Theorem~1]{ElboimKozma2026}.

	For sufficiently large $\beta$, depending on $d$, the expected number of
	visited sites is predicted to have order $n^{\frac{d}{d+1}}$.
	This prediction first appeared in
	\citet*{OrdemannBerkolaikoHavlinBunde2000} for site reinforcement.
	\citet{FosterGrassbergerPaczuski2009} gave numerical and heuristic
	evidence for the same range exponent under site and edge reinforcement.
	In dimension two, the exponent is $2/3$. This growth was also predicted for
	rotor walk on $\Z^2$ with independent uniform initial rotors by
	\citet*{PriezzhevDharDharKrishnamurthy1996}, and this growth was recently proved by
	\citet[Theorem~1.1]{BouRabeePeres2026}.
	In this paper, we prove the corresponding lower bound for once-reinforced random walk: after $n$ steps, the expected number of visited sites is at least $C^{-1}(n/\beta)^{\frac{d}{d+1}}$, uniformly in
	$\beta\geq1$.

	\subsection{Main results}\label{ssec:main-results}

	Fix an integer $d\geq2$ and a reinforcement parameter $\beta\geq1$.  Let
	$(X_n)_{n\geq0}$ be once-reinforced random walk started at the origin of
	$\Z^d$, and let $E_n\coloneqq\{\{X_{i-1},X_i\}:1\leq i\leq n\}$ be the set of
	undirected edges it has crossed in its first $n$ steps, so that
	$E_0=\varnothing$. Given $X_0,\ldots,X_n$, the walk steps from $X_n$ to a neighbor $y$ with
	probability proportional to $\beta$ if $\{X_n,y\}\in E_n$ and proportional
	to $1$ otherwise. We write $\P_\beta$ and
	$\E_\beta$ for the law of the walk and expectation with respect to that law. We call an edge \defn{traversed} once the walk has crossed it, and
	\defn{untraversed} before that.

	The \defn{range} at time $n$ is $R_n\coloneqq\{X_0,\ldots,X_n\}$.

	\begin{theorem}[Range lower bound]\label{thm:main} Let
		$d\geq2$. There is a constant $C\geq1$, depending only on $d$, such that
		for every $\beta\geq1$,
		\begin{equation}\label{eq:main-range}
			\E_\beta|R_n|\geq C^{-1}\beta^{-\frac{d}{d+1}}n^{\frac{d}{d+1}}
			\qquad\text{for every }n\geq0\, .
		\end{equation}
	\end{theorem}
    In Proposition~\ref{prop:beta-sharp} below, we observe that the dependence
    on $\beta$ in \eqref{eq:main-range} is sharp. 

	For comparison, the expected number of sites visited by simple random walk
	on $\Z^d$ by time $n$ is asymptotic to $\pi n/\log n$ when $d=2$, and, when $d\geq3$, to a constant multiple of $n$ \citep[Theorem~1]{DvoretzkyErdos1951}.

    We also obtain a stretched-exponential lower-tail bound. 

	\begin{theorem}[Stretched-exponential tail]\label{thm:whp}
		Let $d\geq2$.  There is a constant $C\geq1$, depending only
		on $d$, such that for every $\beta\geq1$ and every $n$ with
		$n/\beta\geq C$,
		\[
		\P_\beta\{|R_n|<C^{-1}\beta^{-\frac{d}{d+1}}n^{\frac{d}{d+1}}\}
		\leq\exp\bigl(-C^{-1}(n/\beta)^{(d-1)/(d+1)}\bigr)\, .
		\]
	\end{theorem}

	The right-hand side is summable in $n$ for each fixed $\beta$, so the
	Borel--Cantelli lemma gives the almost sure bound
	$\liminf_{n\to\infty}|R_n|n^{-\frac{d}{d+1}}\geq C^{-1}\beta^{-\frac{d}{d+1}}$.

	We also bound the expected time spent at a fixed site. Write
	$L_n(z)\coloneqq\#\{0\leq j<n:X_j=z\}$ for the local time at $z$ before time $n$.

	\begin{theorem}[Occupation measure]\label{thm:occupation-intro}
		Let $d\geq2$.  There is a constant $C<\infty$, depending only on $d$, such
		that for every $\beta\geq1$ and every $n\geq1$,
		\begin{equation}\label{eq:occupation-intro}
			\sup_{z\in\Z^d}\E_\beta L_n(z)\leq C\beta n^{1/(d+1)}\, .
		\end{equation}
	\end{theorem}
    Section~\ref{sec:return} extends this bound to arbitrary deterministic
    time intervals and sharpens its dependence on $\beta$. 
    	
	\begin{figure}[t]
		\centering
		\includegraphics[width=\textwidth]{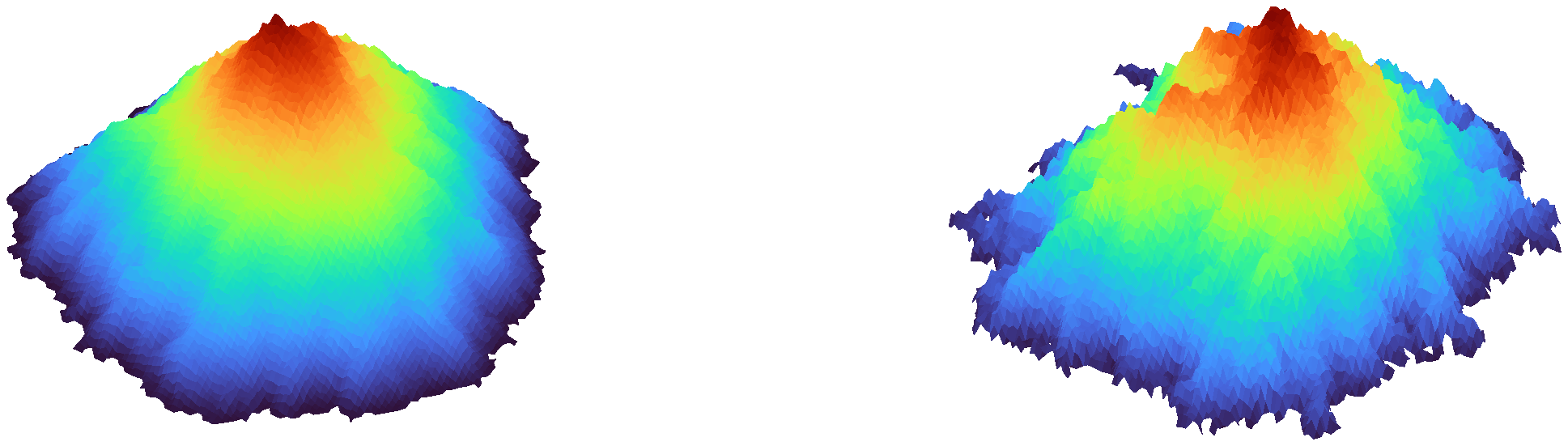}
		\caption{A topographical map of the local time of once-reinforced random walk with $\beta=50$,
		left; and of the clockwise rotor walk with independent uniform initial
		rotors, right. Both walks are on $\Z^2$. Each was run until it had visited
		$70\,000$ sites, taking about $2 \times10^8$ steps on the left and
		$10^6$ on the right.
		The color ranges from dark blue at sites visited once to dark red at the
		site visited most often, $11\,669$ times on the left and $64$ on the right.}
		\label{fig:rotor-comparison}
	\end{figure}
    
	At $d=2$, the bound \eqref{eq:occupation-intro} has order $n^{1/3}$.
	For the clockwise rotor walk on $\Z^2$ with independent uniform initial
	rotors, the range has order $n^{2/3}$ almost surely
	\citep[Theorem~1.1]{BouRabeePeres2026}, so the average local time over
	the visited sites has order $n^{1/3}$.

	\subsection{Proof overview}\label{ssec:proof-overview}

	The heuristic for the exponent $d/(d+1)$ assumes that the walk has
	visited a ball of radius $r$, crossed every edge inside it, and mixed inside
	the ball. The walk then spends a fraction of order $1/r$ of its time at sites
	with an untraversed incident edge. At each such site, the probability
	of crossing an untraversed edge has order $1/\beta$. Adding a site therefore
	takes about $\beta r$ steps. Growing the radius from $r$ to $r+1$ requires
	order $r^{d-1}$ new sites, hence order $\beta r^d$ steps. Reaching radius $r$
	takes order $\beta r^{d+1}$ steps, which predicts
	$r=(n/\beta)^{1/(d+1)}$ and a range of order $(n/\beta)^{\frac{d}{d+1}}$.
    See also the discussion in \citet{Beffara2011}. 
    Figure~\ref{fig:rotor-comparison} shows the time spent at each site in a simulation.

	To prove the range lower bound, we bound the time needed for the walk to cross $k$
	distinct edges. For each integer $k\geq1$, let $\sigma_k$ be the first such time,
	with $\sigma_0=0$.
	We prove
	\[
	\E_\beta\sigma_k\leq C\beta k^{1+1/d}\, ,
	\]
	where $C$ depends only on $d$. Markov's inequality gives the range lower
	bound by taking $k$ of order $(n/\beta)^{\frac{d}{d+1}}$.
	The proof requires no assumption on the shape of the visited set or
	on mixing inside it.

	\medskip
	\noindent\textbf{Step 1. Sum exit times over increasing sets.}
	For a finite set $A$ and $x\in A$, let $2du_A(x)$ be the expected exit time
	of simple random walk from $A$, started at $x$. Extend $u_A$ by zero off
	$A$ and let $u_\varnothing=0$. For $x\notin A$, let $h_A(x)$ be the
	expected time for simple random walk started at $x$ to first return to $x$ or
	leave $A\cup\{x\}$. For every ordering $x_1,\ldots,x_M$ of distinct lattice
	vertices, with $A_i=\{x_1,\ldots,x_i\}$ and $A_0=\varnothing$, we prove
	\begin{equation}\label{eq:overview-insertion}
	\sum_{i=1}^M\bigl(h_{A_{i-1}}(x_i)+u_{A_i}(x_i)\bigr)\leq CM^{1+1/d}\, ,
	\end{equation}
	in Section~\ref{sec:insertion}. Bounding each term separately by
	$CM^{2/d}$ would give $CM^{1+2/d}$ and only the earlier range exponent
	$d/(d+2)$ \citep[Theorem~1.1]{CollevecchioTarres2025}.

	To prove the smaller bound, consider the total normalized exit time
	$T(A)=\sum_{x\in A}u_A(x)$. Adding $x$ to $A$ increases $T$ by
	$h_A(x)u_{A\cup\{x\}}(x)$. This identity bounds the square of each $h$ term
	by a constant times the corresponding increase in $T$. For the $u$ terms,
	the square equals that increase times the effective resistance from the
	added vertex to the complement of the enlarged set (Lemma~\ref{lem:schur}).
	The increases sum to $T(A_M)\leq CM^{1+2/d}$, and the resistances have sum
	at most $CM$ (Lemma~\ref{lem:packing}). For each $r\geq1$, vertices $x_i$
	with $\dist(x_i,A_i^c)\geq r$ are pairwise at distance at least $r$.
	Packing balls around these vertices gives the resistance bound. Cauchy--Schwarz then gives the order
	$\sqrt{M M^{1+2/d}}=M^{1+1/d}$ in \eqref{eq:overview-insertion}.

	\medskip
	\noindent\textbf{Step 2. Bound expected elapsed time.}
	Let $S_n$ be the set of sites whose incident edges have all been crossed
	by time $n$. Inside $S_n$, all $2d$ edge weights equal $\beta$, so the
	walk is simple random walk. Outside $S_n$, it has conditional probability
	at least $1/(2d\beta)$ of crossing an untraversed edge at its next step.

	The potential $\Phi_n=2du_{S_n}(X_n)$ is the expected time for simple
	random walk to leave $S_n$ from the current site. Given the walk up to time $n$, the expected
	increment $\Phi_{n+1}-\Phi_n$ is $-1$ when $X_n\in S_n$. A first crossing can
	enlarge $S_n$ and increase the potential. We define a nondecreasing
	process $\Gamma_n$, starting at zero and increasing
	only at first crossings, so that $n+\Phi_n-\Gamma_n$ is a supermartingale
	(Lemma~\ref{lem:martingale}). At a first crossing at time $n+1$, the increment of $\Gamma$ is
	$2d\beta h_{S_n}(X_n)$ for excursions from the departure site plus
	the new potential $2du_{S_{n+1}}(X_{n+1})$ at the arrival site. The factor $\beta$ comes from the
	lower bound on the probability of a first crossing. Optional stopping
	and $\Phi_n\geq0$ then give
	\[
	\E_\beta\sigma_k\leq\E_\beta\Gamma_{\sigma_k}\, .
	\]

	\medskip
	\noindent\textbf{Step 3. Compare the walk with a vertex ordering.}
	The sets $S_n$ are random and can gain two vertices at once, whereas
	\eqref{eq:overview-insertion} adds one vertex at a time.
	Lemma~\ref{lem:walk-order} orders the visited sites so that domain
	monotonicity bounds the $h$ and $u$ terms defining $\Gamma$ by the
	corresponding terms in \eqref{eq:overview-insertion}.
	Each site contributes at most $2d$ departure terms and one
	nonzero arrival term. There are at most $k+1$ visited sites at time
	$\sigma_k$, so \eqref{eq:overview-insertion} gives
	\[
	\Gamma_{\sigma_k}\leq C\beta k^{1+1/d}\, .
	\]
	Proposition~\ref{prop:charge} proves this bound. Together with Step~2,
	it gives the expected-time estimate and hence the range lower bound
	(Section~\ref{sec:range}).

	Bounds on the individual increments of $\Gamma$ also give an exponential
	supermartingale, from which we obtain
	\[
	\P_\beta\{\sigma_k>8C\beta k^{1+1/d}\}
	\leq\exp\bigl(-C^{-1}k^{(d-1)/d}\bigr)\, .
	\]
	Taking $k$ of order $(n/\beta)^{\frac{d}{d+1}}$ gives the range lower-tail
	bound in Section~\ref{sec:tail}. In Sections~\ref{sec:displacement}
	and~\ref{sec:sharp}, we deduce displacement and box exit-time bounds and
	show that, in a bound uniform in $\beta$, neither power in \eqref{eq:main-range} can be improved
	while the other is held fixed. The occupation estimate in Section~\ref{sec:return}
	uses a separate argument based on a truncated Green function.

	\subsection{Related work}\label{ssec:related-work}

	\subsubsection{Once-reinforced random walk}
	The previous range lower bound on $\Z^d$ is due to
	\citet[Theorem~1.1]{CollevecchioTarres2025}. For every $\beta>1$, they
	proved that there is $\varepsilon>0$, depending on $d$ and $\beta$, such
	that, for all sufficiently large $n$,
	\[
	\P_\beta\{|R_n|\leq\varepsilon n^{d/(d+2)}\}
	\leq\exp\bigl(-\varepsilon n^{d/(d+2)}\bigr)\, .
	\]
	We improve the range exponent from $d/(d+2)$ to $d/(d+1)$, with the
	smaller lower-tail exponent $(d-1)/(d+1)$
	(Theorems~\ref{thm:main} and~\ref{thm:whp}). Their proof used a
	continuous-time embedding and a change of measure to compare the walk
	with simple random walk. Our proof bounds expected exit times summed
	over increasing sets.

	On $\Z$, once-reinforced random walk is recurrent
	\citep[Theorem~3.2]{Davis1990}, and the asymptotics of every moment of
	$|R_n|/\sqrt n$ are known for every positive weight $\beta$ on traversed edges
	\citep[Theorem~1]{PfaffelhuberStiefel2021}.
	Recurrence holds on the doubly infinite ladder for every $\beta>1$
	\citep[Theorem~2.1]{Sellke2006}, and on $\Z$ times a finite connected
	graph for sufficiently large $\beta$. In the latter setting,
	\citet{KiousSchapiraSingh2021} also proved a shape theorem for the range.

	On regular trees of degree at least three, the walk is transient with
	positive speed for every $\beta>1$, and its distance from the root
	satisfies an invariance principle \citep{DurrettKestenLimic2002}.
	In contrast, \citet{KiousSidoravicius2018} constructed trees of polynomial
	growth on which increasing reinforcement changes transience to
	recurrence. \citet[Theorem~1]{CollevecchioKiousSidoravicius2020}
	identified the critical reinforcement on every infinite locally finite
	tree and classified recurrence and transience away from that value.
	For $\Z^d$ with $d\geq3$, Beffara and Sidoravicius conjectured a phase
	transition in $\beta$ \citep[Section~2]{Kozma2012}.

	\subsubsection{Site reinforcement}
	In the physics literature, \citet{Sapozhnikov1994} independently
	introduced the once-site-reinforced self-attracting walk.
	Early simulations were not precise enough to determine the displacement
	exponent or whether the walk undergoes a phase transition as the
	reinforcement varies. For sufficiently large reinforcement,
	\citet*{OrdemannBerkolaikoHavlinBunde2000} predicted root-mean-square
	displacement of order $n^{1/(d+1)}$ and expected range of order
	$n^{d/(d+1)}$; see also \citet*{OrdemannTomerBerkolaikoHavlinBunde2001}.
	\citet*{FosterGrassbergerPaczuski2009} supported these predictions with
	further simulations and heuristic arguments, and predicted the same
	exponents for edge reinforcement.

	For several variants of once-reinforced walk  on $\Z^2$,
	\citet[Conjecture~7]{Beffara2011} predicted recurrence at sufficiently
	large reinforcement, range size of order $t^{2/3}$, diameter of order
	$t^{1/3}$, and convergence in probability to a deterministic shape
	after rescaling by the diameter. 

	\subsubsection{Rotor walk}
	The same range lower bound holds for rotor walk on $\Z^d$:
	\citet[Theorem~1.1]{FlorescuLevinePeres2016} proved that it visits at least
	a constant multiple of $n^{d/(d+1)}$ sites in $n$ steps, for every rotor
	mechanism and initial configuration. On $\Z^2$ with the clockwise
	mechanism and independent uniform initial rotors,
	\citet[Theorem~1.1]{BouRabeePeres2026} proved that the range size divided
	by $n^{2/3}$ converges almost surely to a deterministic positive constant.

	\subsection{Notation and conventions}

	\begin{itemize}
		\item $x\sim y$ means that $x$ and $y$ are nearest neighbors in $\Z^d$, and
		$\dist(x,D)$ is the least number of edges in a path from $x$ to a nonempty set
		$D\subset\Z^d$, which is the $\ell^1$ distance from $x$ to $D$.
		We write $\dist(x,y)$ for $\dist(x,\{y\})$.
		\item $|D|$ is the cardinality of a finite set $D$, and $\|\cdot\|_\infty$ is
		the $\ell^\infty$ norm on $\Z^d$.
		\item $C\geq1$ denotes a finite constant depending only on $d$ that may
		change from line to line. We write $C^{-1}$ for the corresponding lower bounds.
	\end{itemize}

	\subsection{Formalization}\label{ssec:formalization}

	Every theorem, lemma and proposition of this paper has been formalized in
	Lean~4 and proved from Mathlib, together with the almost sure bounds stated
	after Theorem~\ref{thm:whp} and in Proposition~\ref{prop:displacement}.  No
	proof uses an axiom beyond the three of Lean's own logic, and no proof uses
	\texttt{sorry}.  The development is at
	\url{https://github.com/nitromannitol/ORRW-Lower-Bound}.

	The formalization defines effective resistance by \eqref{eq:schur};
	it does not prove the electrical-network interpretation.

	\section*{Acknowledgments}
	The research of Y. Peres was supported by National Natural Science Foundation of China grant
	RFIS-W2531011.  We acknowledge use of ChatGPT 5.6.

	\section{Exit times for increasing sets}\label{sec:insertion}

	Let $(Y_n)_{n\geq0}$ be simple random walk on $\Z^d$, with law $\P_x$ and
	expectation $\E_x$ when started at $x$. For a finite set $A$, let
	$\tau_A=\inf\{n\geq0:Y_n\notin A\}$ be its exit time. The
	\defn{Dirichlet Laplacian} $\Delta_A$ acts on functions supported in $A$ by
	\[
	(\Delta_Af)(x)\coloneqq\sum_{y\in A,\ y\sim x}f(y)-2df(x)
	\qquad(x\in A)\, .
	\]
	Thus $(2d)^{-1}\Delta_A$ is the generator of the walk killed on leaving
	$A$. Extending $f$ by zero and summing over undirected edges gives
	\[
	\langle f,\Delta_Af\rangle=-\sum_{\{x,y\}}\bigl(f(x)-f(y)\bigr)^2\, .
	\]
	A nonzero function supported in a finite set differs across some edge,
	so $\Delta_A$ is symmetric and negative definite, hence invertible.

	Define the normalized exit time $u_A$ and its sum $T(A)$ by
	\[
	u_A\coloneqq-\Delta_A^{-1}\one_A,
	\qquad T(A)\coloneqq\sum_{x\in A}u_A(x)\, ,
	\]
	with $u_A=0$ off $A$, $u_\varnothing=0$, and $T(\varnothing)=0$.
	First-step analysis gives
	\begin{equation}\label{eq:mean-steps}
		2du_A(x)=\E_x\tau_A
		\qquad(x\in A)\, .
	\end{equation}
	For $x\notin A$, let
	\begin{equation}\label{eq:h-def}
		h_A(x)\coloneqq1+\sum_{y\in A,\ y\sim x}u_A(y)\, .
	\end{equation}
	This is the expected time for simple random walk started at $x$ to return
	to $x$ for the first time or leave $A\cup\{x\}$. The first step takes one unit of time.
	If it lands at $y\in A$, the expected time then spent in $A$ is $2du_A(y)$.

	We will also use domain monotonicity. If $A\subseteq D$, then
	$\tau_A\leq\tau_D$ along the same walk started in $A$. Together with
	nonnegativity and zero extension, \eqref{eq:mean-steps} and \eqref{eq:h-def}
	give
	\[
	u_A\leq u_D\ \text{on }A
	\quad\text{and}\quad
	h_A(x)\leq h_D(x)\ \text{whenever }A\subseteq D\text{ and }x\notin D\, .
	\]

	\begin{lemma}[Exit time bounds]\label{lem:max-exit}
		Let $d\geq1$. Every finite nonempty $A\subset\Z^d$ satisfies
		\[
		\max_{x\in A}u_A(x)\leq C|A|^{2/d},
		\qquad
		\max_{x\notin A}h_A(x)\leq C|A|^{2/d},
		\qquad
		T(A)\leq C|A|^{1+2/d}\, .
		\]
	\end{lemma}

	\begin{proof}
		The Nash inequality on $\Z^d$ with exponent $d$
		\citep[Lemma~3.13]{Barlow2017} gives the Faber--Krahn inequality
		$\lambda(A)\geq C^{-1}|A|^{-2/d}$ for the smallest eigenvalue of $-\Delta_A$.
		On the unit-conductance lattice $A$ has weighted volume $2d|A|$ and the
		normalized Dirichlet Laplacian is $\Delta_A/(2d)$, so the mean number of steps needed
		to leave $A$ is at most $C|A|^{2/d}$
		\citep[Proposition~3.5]{Telcs2006}, and that mean
		is $2du_A$ by \eqref{eq:mean-steps}.  The bound for $h_A$ follows from
		\eqref{eq:h-def} and the degree bound, and the bound for $T(A)$ by summing
		over $A$.
	\end{proof}

	\begin{lemma}[Adding one vertex]\label{lem:schur}
		Let $d\geq1$, let $A\subset\Z^d$ be finite, let $x\notin A$, and set
		$A^+\coloneqq A\cup\{x\}$.  Give every edge of $\Z^d$ unit conductance and
		let $\Reff\bigl(x\leftrightarrow(A^+)^c\bigr)$ be the effective resistance
		between $x$ and the grounded complement of $A^+$.  Then
		\begin{equation}\label{eq:schur}
			\Reff\bigl(x\leftrightarrow(A^+)^c\bigr)=-\Delta_{A^+}^{-1}(x,x)\geq\frac1{2d},
			\qquad
			u_{A^+}(x)=\Reff\bigl(x\leftrightarrow(A^+)^c\bigr)h_A(x)\, ,
		\end{equation}
		and
		\begin{equation}\label{eq:deltaT}
			T(A^+)-T(A)=h_A(x)u_{A^+}(x)\, .
		\end{equation}
		Consequently
		\begin{equation}\label{eq:hbound}
			h_A(x)^2\leq2d\bigl(T(A^+)-T(A)\bigr),
			\qquad
			u_{A^+}(x)^2
			=\Reff\bigl(x\leftrightarrow(A^+)^c\bigr)\bigl(T(A^+)-T(A)\bigr)\, .
		\end{equation}
	\end{lemma}

	\begin{proof}
		Write $R\coloneqq\Reff\bigl(x\leftrightarrow(A^+)^c\bigr)$.  Before leaving
		$A^+$, simple random walk started at $x$ visits $x$ a geometric number
		of times with mean $2dR$
		\citep[Equation~(2.5)]{LyonsPeres2016}, and the expected time from
		each visit until the next return or exit is $h_A(x)$, so
		$\E_x\tau_{A^+}=2dRh_A(x)$ and \eqref{eq:mean-steps} gives
		\eqref{eq:schur}, the inequality coming from the visit at time zero.
		The strong Markov property at $\tau_A$ gives
		$u_{A^+}(y)=u_A(y)-\Delta_{A^+}^{-1}(y,x)u_{A^+}(x)/R$ for $y\in A$, and
		$-\Delta_{A^+}^{-1}$ is symmetric with row sums $u_{A^+}$, so summing over
		$A^+$ gives \eqref{eq:deltaT} and \eqref{eq:hbound} follows from
		\eqref{eq:schur}.
	\end{proof}

	\begin{lemma}[Sum of effective resistances]\label{lem:packing}
		Let $d\geq2$, let $M\geq1$ be an integer, let $x_1,\ldots,x_M$ be distinct
		vertices of $\Z^d$, and let $A_i\coloneqq\{x_1,\ldots,x_i\}$ for
		$1\leq i\leq M$.  Then
		\[
			\sum_{i=1}^M\Reff\bigl(x_i\leftrightarrow A_i^c\bigr)\leq CM\, .
		\]
	\end{lemma}

	\begin{proof}
		Let $\rho_i\coloneqq\dist(x_i,A_i^c)$.  Sending a unit flow along a shortest
		path to $A_i^c$, Thomson's principle
		\citep[Section~1.3.5]{DoyleSnell1984} gives
		\begin{equation}\label{eq:thomson}
			\Reff\bigl(x_i\leftrightarrow A_i^c\bigr)\leq\rho_i\, .
		\end{equation}
		For each integer $r\geq2$, the sites with $\rho_i\geq r$ are separated
		by distance at least $r$: if $i<j$, then $x_j\notin A_i$, whereas every
		vertex within distance $r-1$ of $x_i$ lies in $A_i$.
		Their $\ell^1$-balls of radius $\lfloor(r-1)/2\rfloor$ are therefore
		disjoint subsets of $A_M$, each containing at least $C^{-1}r^d$ vertices.
		Thus $|\{i:\rho_i\geq r\}|\leq CMr^{-d}$, and
		\[
		\sum_{i=1}^M\Reff\bigl(x_i\leftrightarrow A_i^c\bigr)
		\leq\sum_{r\geq1}|\{i:\rho_i\geq r\}|
		\leq CM\sum_{r\geq1}r^{-d}\leq CM\, ,
		\]
		since $d\geq2$.
	\end{proof}

	\begin{lemma}[Sum over successive vertex additions]\label{lem:insertion}
		Let $d\geq2$, let $M\geq1$ be an integer, let $x_1,\ldots,x_M$ be distinct
		vertices of $\Z^d$, and let $A_i\coloneqq\{x_1,\ldots,x_i\}$ for
		$0\leq i\leq M$.  Then
		\begin{equation}\label{eq:insertion-sum}
			\sum_{i=1}^M\bigl(h_{A_{i-1}}(x_i)+u_{A_i}(x_i)\bigr)
			\leq CM^{1+1/d}\, .
		\end{equation}
	\end{lemma}

	\begin{proof}
		The first relation in \eqref{eq:hbound}, Lemma~\ref{lem:max-exit}, and the
		Cauchy--Schwarz inequality give
		\[
		\sum_{i=1}^Mh_{A_{i-1}}(x_i)^2\leq2dT(A_M)\leq CM^{1+2/d},
		\qquad
		\sum_{i=1}^Mh_{A_{i-1}}(x_i)\leq CM^{1+1/d}\, .
		\]

		Write $\Reff_i\coloneqq\Reff(x_i\leftrightarrow A_i^c)$. By the second
		identity in \eqref{eq:hbound},
		$u_{A_i}(x_i)^2=\Reff_i(T(A_i)-T(A_{i-1}))$.
		Cauchy--Schwarz and Lemmas~\ref{lem:packing} and~\ref{lem:max-exit} give
		\[
		\sum_{i=1}^Mu_{A_i}(x_i)
		\leq\Bigl(\sum_{i=1}^M\Reff_i\Bigr)^{1/2}T(A_M)^{1/2}
		\leq CM^{1/2}\bigl(M^{1+2/d}\bigr)^{1/2}=CM^{1+1/d}\, .
		\]
		Adding the last two estimates proves \eqref{eq:insertion-sum}.
	\end{proof}

	\section{A potential and a supermartingale}\label{sec:potential}

	Let
	\[
		\begin{gathered}
		S_n\coloneqq\{x\in\Z^d:\text{every edge incident to $x$ lies in $E_n$}\},\\
		\nu_n\coloneqq\#\{y\sim X_n:\{X_n,y\}\notin E_n\},
		\qquad
		\mu_n\coloneqq\beta(2d-\nu_n)+\nu_n\, .
		\end{gathered}
	\]
	Here $S_0=\varnothing$, $\nu_n$ counts the untraversed edges incident to
	$X_n$, and $\mu_n$ is the total weight of the edges incident to $X_n$.
	We have $\mu_n\leq2d\beta$ and $X_n\in S_n$ if and only if $\nu_n=0$. Let
	\[
		\sigma_k\coloneqq\inf\{n\geq0:|E_n|=k\},
		\qquad\sigma_0\coloneqq0,
		\qquad v_k\coloneqq X_{\sigma_k-1},
		\qquad w_k\coloneqq X_{\sigma_k}\, ,
	\]
	for every integer $k\geq1$, with $v_k$ and $w_k$ defined on
	$\{\sigma_k<\infty\}$. Thus $\sigma_k$ is the time of the $k$th first
	crossing, and $v_k$ and $w_k$ are its departure and arrival sites.  Since $|E_n|$ increases by at most one at each step,
	\begin{equation}\label{eq:sigma-range}
		\{\sigma_k\leq n\}=\{|E_n|\geq k\}\, .
	\end{equation}
	For $n\geq1$, write $\eta_n\coloneqq\one_{\{|E_n|>|E_{n-1}|\}}$ for the indicator that the
	$n$th step crosses an untraversed edge, and let $(\mathcal F_n)_{n\geq0}$ be
	the filtration generated by $(X_n)_{n\geq0}$.

	The \defn{potential} is
	\begin{equation}\label{eq:F-def}
		\Phi_n\coloneqq2du_{S_n}(X_n)\, ,
	\end{equation}
	which vanishes when $X_n\notin S_n$.  By \eqref{eq:mean-steps} it is the
	expected number of steps that simple random walk started at $X_n$ needs to
	leave $S_n$.  When $\nu_n\geq1$, so that $X_n\notin S_n$ and $h_{S_n}(X_n)$ is
	defined, let
	\begin{equation}\label{eq:a-def}
		\gamma_n\coloneqq2d\beta h_{S_n}(X_n)\, ,
	\end{equation}
	and let $\gamma_n\coloneqq0$ when $\nu_n=0$.  Define the nondecreasing process $(\Gamma_n)_{n\geq0}$ by
	\begin{equation}\label{eq:C-def}
		\Gamma_0\coloneqq0,
		\qquad
		\Gamma_{n+1}-\Gamma_n\coloneqq\eta_{n+1}\bigl(\gamma_n+\Phi_{n+1}\bigr)\, .
	\end{equation}
	The process $\Gamma_n$ increases only at the steps $\sigma_1,\sigma_2,\ldots$,
	with increment $\gamma_{\sigma_j-1}+\Phi_{\sigma_j}$ at step $\sigma_j$.

	\begin{lemma}[A supermartingale]\label{lem:martingale}
		Let $d\geq1$ and $\beta\geq1$.  Then $(n+\Phi_n-\Gamma_n)_{n\geq0}$ is a
		supermartingale for the filtration $(\mathcal F_n)_{n\geq0}$.
	\end{lemma}

	\begin{proof}
		Lemma~\ref{lem:max-exit} and $|S_n|\leq n$ give integrability of
		$\Phi_n$ and $\Gamma_n$.  It remains to prove
		\begin{equation}\label{eq:drift}
			\E_\beta\bigl[\Gamma_{n+1}-\Gamma_n-(\Phi_{n+1}-\Phi_n)\mid\mathcal F_n\bigr]
			\geq1\, .
		\end{equation}
		If $X_n\in S_n$, then, conditional on $\mathcal F_n$, the next site is
		uniform among the neighbors of $X_n$, with
		$S_{n+1}=S_n$ and $\Gamma_{n+1}=\Gamma_n$.  The exit-time identity
		\eqref{eq:mean-steps} gives
		$\E_\beta[\Phi_{n+1}\mid\mathcal F_n]=\Phi_n-1$, so equality holds.

		If $X_n\notin S_n$, then $\Phi_n=0$ and $\nu_n\geq1$.
		When an untraversed edge is crossed, $\Phi_{n+1}$ also appears in the
		increment of $\Gamma_n$, so its contributions to \eqref{eq:drift} cancel.
		Each neighbor in $S_n$ is reached with probability $\beta/\mu_n$.
		Writing $h\coloneqq h_{S_n}(X_n)$, equations~\eqref{eq:h-def} and
		\eqref{eq:a-def} show that the left side of \eqref{eq:drift} is
		\[
		\frac{\nu_n\gamma_n}{\mu_n}-\frac{2d\beta}{\mu_n}(h-1)
		=\frac{2d\beta}{\mu_n}\bigl(h(\nu_n-1)+1\bigr)\geq1\, ,
		\]
		since $h\geq1$ and $\mu_n\leq2d\beta$.
	\end{proof}

	The sets $S_{\sigma_j}$ are random, and the bound in
	\eqref{eq:insertion-sum} holds for every ordering of a finite set of sites.  The next lemma gives, for each realization of the walk, an ordering
	of the sites it has visited for which each $S_{\sigma_j}$ is contained in a
	set $A_i$ of the form $A_i=\{x_1,\ldots,x_i\}$.

	\begin{lemma}[Ordering the visited sites]\label{lem:walk-order}
		Let $k\geq1$ be an integer, and work on the event $\{\sigma_k<\infty\}$.
		Let $V\coloneqq R_{\sigma_k}$ and $M\coloneqq|V|$.  Then $M\leq k+1$, and
		there is an ordering $x_1,\ldots,x_M$ of $V$ for which the sets
		$A_i\coloneqq\{x_1,\ldots,x_i\}$ satisfy the following for every
		$1\leq j\leq k$.
		\begin{enumerate}[label=\textup{(\roman*)}]
			\item If $v_j=x_i$, then $S_{\sigma_{j-1}}\subseteq A_{i-1}$.
			\item If $w_j\in S_{\sigma_j}$ and $w_j=x_i$, then
			$S_{\sigma_j}\subseteq A_i$.  The sites $w_j$ satisfying
			$w_j\in S_{\sigma_j}$ are distinct.
		\end{enumerate}
	\end{lemma}

	\begin{proof}
		The edges of $E_{\sigma_k}$ form a connected subgraph with $k$ edges whose
		vertex set is $V$, so $M\leq k+1$.  Every $x\in V$ is an endpoint of some
		edge of $E_{\sigma_k}$; let $\iota(x)$ be the largest $j\leq k$ for which
		$x$ is an endpoint of the $j$th distinct edge the walk crosses.  Only the
		endpoints of that edge can have $\iota=j$, so at most two sites satisfy $\iota(x)=j$ for each
		$j$.  Order $V$ by increasing $\iota$, placing $v_j$ before $w_j$ whenever
		$\iota(v_j)=\iota(w_j)=j$.

		For (i), let $y\in S_{\sigma_{j-1}}$.  Every edge at $y$ has been crossed by
		step $\sigma_{j-1}$, so $\iota(y)\leq j-1$, whereas the $j$th distinct edge
		is incident to $v_j$ and hence $\iota(v_j)\geq j$.  Thus $y$ precedes $v_j$,
		which gives $S_{\sigma_{j-1}}\subseteq A_{i-1}$ when $v_j=x_i$.

		For (ii), suppose that $w_j\in S_{\sigma_j}$.  The $j$th distinct edge is
		incident to $w_j$ and is untraversed before step $\sigma_j$, so
		$w_j\notin S_{\sigma_{j-1}}$ and therefore $\iota(w_j)=j$.  Every
		$y\in S_{\sigma_j}$ has $\iota(y)\leq j$, and the tie-breaking rule places
		the other site with $\iota=j$, if there is one, before $w_j$.  Hence
		$S_{\sigma_j}\subseteq A_i$ when $w_j=x_i$.  Finally, distinct indices $j$
		give distinct values $\iota(w_j)=j$ and hence distinct sites $w_j$.
	\end{proof}

	\begin{proposition}[Bound on the nondecreasing process]\label{prop:charge}
		Let $d\geq2$.  There is a constant $C<\infty$, depending only on $d$, such
		that for every $\beta\geq1$, almost surely,
		\begin{equation}\label{eq:charge-bound}
			\Gamma_n\leq C\beta|E_n|^{1+1/d}
			\qquad\text{for every }n\geq0\, .
		\end{equation}
	\end{proposition}

	\begin{proof}
		Fix $n$ and let $k\coloneqq|E_n|$.  If $k=0$ then $\Gamma_n=0$, so assume
		$k\geq1$; then $\sigma_k\leq n<\infty$.  No new edge is crossed between
		steps $\sigma_{j-1}$ and $\sigma_j-1$, so $S_{\sigma_j-1}=S_{\sigma_{j-1}}$,
		and $X_{\sigma_j-1}=v_j$ and $X_{\sigma_j}=w_j$.  Hence \eqref{eq:C-def},
		\eqref{eq:a-def} and \eqref{eq:F-def} give
		\[
			\Gamma_n=\sum_{j=1}^k\bigl(\gamma_{\sigma_j-1}+\Phi_{\sigma_j}\bigr)
			=2d\beta\sum_{j=1}^kh_{S_{\sigma_{j-1}}}(v_j)
			+2d\sum_{j=1}^ku_{S_{\sigma_j}}(w_j)\, .
		\]

		Take the ordering $x_1,\ldots,x_M$ of Lemma~\ref{lem:walk-order}, so that
		$M\leq k+1$.  A site occurs as $v_j$ for at most $2d$ indices $j$, because
		those indices correspond to distinct edges incident to that site.  The site $v_j$ has an
		untraversed incident edge at step $\sigma_j-1$, so
		$v_j\notin S_{\sigma_{j-1}}$, and part~(i) of the lemma together with domain
		monotonicity gives $h_{S_{\sigma_{j-1}}}(v_j)\leq h_{A_{i-1}}(x_i)$ when
		$v_j=x_i$.  The term $u_{S_{\sigma_j}}(w_j)$ vanishes unless
		$w_j\in S_{\sigma_j}$, and in that case part~(ii) and domain monotonicity
		give $u_{S_{\sigma_j}}(w_j)\leq u_{A_i}(x_i)$ at indices $i$ that are
		distinct for distinct $j$.  Using $\beta\geq1$ and \eqref{eq:insertion-sum},
		\[
			\Gamma_n\leq4d^2\beta\sum_{i=1}^M\bigl(h_{A_{i-1}}(x_i)+u_{A_i}(x_i)\bigr)
			\leq C\beta M^{1+1/d}\leq C\beta(2k)^{1+1/d}\, ,
		\]
		which is \eqref{eq:charge-bound} for a suitable $C$.
	\end{proof}

	\section{The expected range}\label{sec:range}
	\begin{proof}[Proof of Theorem~\ref{thm:main}]
		For integers $k\geq1$ and $N\geq0$, apply optional stopping to
		Lemma~\ref{lem:martingale} at $\sigma_k\wedge N$.  Since $\Phi\geq0$ and
		$|E_{\sigma_k\wedge N}|\leq k$, Proposition~\ref{prop:charge} gives
		$\E_\beta[\sigma_k\wedge N]\leq C\beta k^{1+1/d}$.
		Monotone convergence as $N\to\infty$ gives
		\begin{equation}\label{eq:sigma-mean}
			\E_\beta\sigma_k\leq C\beta k^{1+1/d}\, .
		\end{equation}

		Set $a\coloneqq(n/(2C\beta))^{\frac{d}{d+1}}$.
		If $a\geq1$, take $k=\lfloor a\rfloor\geq a/2$.
		Markov's inequality and \eqref{eq:sigma-mean} give
		$\P_\beta\{\sigma_k\leq n\}\geq1/2$.
		The degree bound gives
		\begin{equation}\label{eq:edges-sites}
			|E_n|\leq d|R_n|\, .
		\end{equation}
		Together with \eqref{eq:sigma-range}, this gives
		\[
		\E_\beta|R_n|\geq\frac{k}{2d}\geq\frac a{4d}
		\geq C^{-1}\beta^{-\frac{d}{d+1}}n^{\frac{d}{d+1}}\, .
		\]
		For $a<1$, use $|R_n|\geq1$ and increase $C$.
	\end{proof}

	\section{A stretched-exponential tail}\label{sec:tail}

	An exponential supermartingale strengthens the expected-time estimate
	\eqref{eq:sigma-mean} to a tail bound for $\sigma_k$.

	\begin{lemma}[An exponential-moment bound]\label{lem:exp}
		Let $d\geq1$ be an integer and let $\beta\geq1$.  Let
		$(\mathcal F_n)_{n\geq0}$ be the filtration generated by once-reinforced
		random walk on $\Z^d$ with reinforcement $\beta$.  Let $(\Phi_n)_{n\geq0}$ and
		$(\Gamma_n)_{n\geq0}$ be adapted processes with $\Phi_0=\Gamma_0=0$, and let
		$\zeta$ be a stopping time for this filtration.  Suppose that
		$(n\wedge\zeta+\Phi_{n\wedge\zeta}-\Gamma_{n\wedge\zeta})_{n\geq0}$ is a
		supermartingale.
		Suppose that $p$, $q$ and $b$ are finite constants with
		\[
			0\leq\Phi_{n\wedge\zeta}\leq p,
			\qquad
			0\leq\Gamma_{(n+1)\wedge\zeta}-\Gamma_{n\wedge\zeta}\leq q,
			\qquad
			\Gamma_{n\wedge\zeta}\leq b\, ,
		\]
		for every $n\geq0$, and let
		$\theta\coloneqq\bigl[4(1+p+q)\bigr]^{-1}$.  Then for every integer $N\geq0$,
		\[
			\E_\beta e^{\theta(N\wedge\zeta)/2}\leq e^{2\theta b},
			\qquad\text{and hence}\qquad
			\P_\beta\{N\wedge\zeta\geq4b+s\}\leq e^{-\theta s/2}
			\quad\text{for every }s\geq0\, .
		\]
	\end{lemma}

	\begin{proof}
		Write $\delta_n\coloneqq\Gamma_{(n+1)\wedge\zeta}-\Gamma_{n\wedge\zeta}$
		and set
		\[
		Z_n\coloneqq e^{\theta(n\wedge\zeta)/2-2\theta\Gamma_{n\wedge\zeta}}
		\bigl(1+\theta\Phi_{n\wedge\zeta}\bigr)\, .
		\]
		We show that $(Z_n)$ is a supermartingale.
		Since $0\leq\theta\delta_n\leq1/4$, the inequality
		$e^{-2x}\leq1-x$ for $0\leq x\leq1/2$ and the assumed supermartingale
		property give
		\begin{align*}
		\E_\beta\bigl[e^{-2\theta\delta_n}
		(1+\theta\Phi_{(n+1)\wedge\zeta})\mid\mathcal F_n\bigr]
		&\leq1+\theta\Phi_{n\wedge\zeta}-\theta\one_{\{n<\zeta\}}\\
		&\leq e^{-\theta\one_{\{n<\zeta\}}/2}
		(1+\theta\Phi_{n\wedge\zeta})\, .
		\end{align*}
		The last inequality uses $1+\theta\Phi_{n\wedge\zeta}\leq2$ and
		$1-x\leq e^{-x}$.  Since
		$((n+1)\wedge\zeta)-(n\wedge\zeta)=\one_{\{n<\zeta\}}$, it follows that
		$\E_\beta[Z_{n+1}\mid\mathcal F_n]\leq Z_n$.
		Now $Z_0=1$, $\Phi_{N\wedge\zeta}\geq0$, and $\Gamma_{N\wedge\zeta}\leq b$ give
		$\E_\beta e^{\theta(N\wedge\zeta)/2}\leq e^{2\theta b}$.
		The tail bound follows by Markov's inequality.
	\end{proof}

	\begin{proposition}[Tail of $\sigma_k$]\label{prop:tail}
		Let $d\geq2$. There is a constant $C\geq1$, depending only on $d$, such
		that for every $\beta\geq1$ and every integer $k\geq1$,
		\[
			\P_\beta\bigl\{\sigma_k>8C\beta k^{1+1/d}\bigr\}
			\leq\exp\bigl(-C^{-1}k^{(d-1)/d}\bigr)\, .
		\]
	\end{proposition}

	\begin{proof}
		Apply Lemma~\ref{lem:exp} with $\zeta=\sigma_k$ to the supermartingale
		of Lemma~\ref{lem:martingale}.  For $n\leq\sigma_k$, we have $|S_n|\leq k/d$.
		Lemma~\ref{lem:max-exit} and \eqref{eq:C-def} therefore bound
		$\Phi_{n\wedge\sigma_k}$ by $Ck^{2/d}$ and
		$\Gamma_{(n+1)\wedge\sigma_k}-\Gamma_{n\wedge\sigma_k}$ by $2C\beta k^{2/d}$.
		Choose $C\geq8$ large enough for these bounds and
		Proposition~\ref{prop:charge}. In Lemma~\ref{lem:exp}, take
		\[
		p=Ck^{2/d},\qquad q=2C\beta k^{2/d},\qquad
		b=C\beta k^{1+1/d}\, .
		\]
		Then $\theta=[4(1+p+q)]^{-1}\geq1/(16C\beta k^{2/d})$.
		Taking $s=4b$ and an integer $N>8b$ gives
		\[
		\P_\beta\{\sigma_k>8b\}
		\leq\P_\beta\{N\wedge\sigma_k\geq8b\}
		\leq e^{-2\theta b}\leq\exp\bigl(-k^{(d-1)/d}/8\bigr)
		\leq\exp\bigl(-C^{-1}k^{(d-1)/d}\bigr)\, .
		\]
	\end{proof}

	\begin{proof}[Proof of Theorem~\ref{thm:whp}]
		Suppose that $n\geq16C\beta$, and let
		\[
		k\coloneqq\lfloor(n/(8C\beta))^{\frac{d}{d+1}}\rfloor\, .
		\]
		Then $(n/(8C\beta))^{\frac{d}{d+1}}\geq2^{\frac{d}{d+1}}\geq1$, so $k\geq1$ and
		\[
		k\geq\tfrac12(n/(8C\beta))^{\frac{d}{d+1}},
		\qquad 8C\beta k^{1+1/d}\leq n\, .
		\]
		By \eqref{eq:sigma-range} and
		Proposition~\ref{prop:tail},
		\[
			\P_\beta\{|E_n|<k\}
			=\P_\beta\{\sigma_k>n\}
			\leq\P_\beta\bigl\{\sigma_k>8C\beta k^{1+1/d}\bigr\}
			\leq\exp\bigl(-C^{-1}k^{(d-1)/d}\bigr)\, .
		\]
		On $\{|E_n|\geq k\}$, equation~\eqref{eq:edges-sites} gives
		\[
		|R_n|\geq k/d\geq C^{-1}(n/\beta)^{\frac{d}{d+1}},
		\qquad k^{(d-1)/d}\geq C^{-1}(n/\beta)^{(d-1)/(d+1)}\, .
		\]
		The theorem follows after
		enlarging $C$.
	\end{proof}

	\section{Displacement and exit times}\label{sec:displacement}

	Let $B_r\coloneqq[-r,r]^d\cap\Z^d$ and let
	$\tau_{B_r}\coloneqq\inf\{n\geq0:X_n\notin B_r\}$ be its exit time.
	Define the maximal displacement by
	$\Lambda_n\coloneqq\max_{0\leq k\leq n}\|X_k\|_\infty$.
	Since $R_n\subseteq B_{\Lambda_n}$,
	\begin{equation}\label{eq:range-box}
		|R_n|\leq(2\Lambda_n+1)^d,
		\qquad\text{equivalently}\qquad
		\Lambda_n\geq\tfrac12\bigl(|R_n|^{1/d}-1\bigr)\, .
	\end{equation}

	\begin{proposition}[Displacement and exit times]\label{prop:displacement}
		Let $d\geq2$.  There is a constant $C\geq1$, depending only
		on $d$, such that for every $\beta\geq1$,
		\[
		\E_\beta\tau_{B_r}\leq C\beta(2r+1)^{d+1}
		\qquad\text{for every integer }r\geq0\, ,
		\]
		and
		\[
		\E_\beta\Lambda_n\geq C^{-1}\beta^{-1/(d+1)}n^{1/(d+1)}
		\qquad\text{for every }n\geq0\, .
		\]
		Moreover, $\P_\beta$-almost surely, \[
		\liminf_{n\to\infty}\Lambda_nn^{-1/(d+1)}\geq C^{-1}\beta^{-1/(d+1)}\, . \]
	\end{proposition}

	\begin{proof}
		The box $B_r$ contains at most $d(2r+1)^d$ edges, so
		$\tau_{B_r}\leq\sigma_{k_r}$ for $k_r=d(2r+1)^d+1$.
		Equation~\eqref{eq:sigma-mean} gives
		\[
		\E_\beta\tau_{B_r}\leq C\beta k_r^{1+1/d}
		\leq C\beta(2r+1)^{d+1}\, .
		\]

		For the expected maximal displacement, set
		$a=(n/(2C\beta))^{\frac{d}{d+1}}$ and $k=\lfloor a\rfloor$.
		If $a\geq2^{d+1}d$, the proof of Theorem~\ref{thm:main} gives
		$\P_\beta\{|R_n|\geq k/d\}\geq1/2$ and $k\geq a/2$.
		On this event, \eqref{eq:range-box} gives
		$\Lambda_n\geq\tfrac14(k/d)^{1/d}$, and hence
		\[
		\E_\beta\Lambda_n\geq\tfrac18(a/(2d))^{1/d}
		\geq C^{-1}\beta^{-1/(d+1)}n^{1/(d+1)}\, .
		\]
		For $a<2^{d+1}d$, use $\Lambda_n\geq1$ when $n\geq1$ and increase $C$;
		at $n=0$ both sides vanish.
		The almost sure bound follows from \eqref{eq:range-box} and the almost
		sure range bound in Section~\ref{ssec:main-results}.
	\end{proof}

	\section{Optimality of the exponents}\label{sec:sharp}

	Taking $\beta=n$ keeps the expected range bounded as $n\to\infty$.
	Consequently, a lower bound uniform in $\beta$ cannot have a larger power
	of $n$ with the power of $\beta$ fixed, or a larger power of $\beta$ with
	the power of $n$ fixed. Proposition~\ref{prop:beta-sharp} gives the precise
	statement.

	\begin{proposition}[The exponents are optimal]\label{prop:beta-sharp} Let
		$d\geq2$. For every $\beta\geq1$ and every $n\geq1$,
		\begin{equation}\label{eq:range-upper}
			\E_\beta|R_n|\leq2+2dn/\beta\, .
		\end{equation}
		Consequently, if the real numbers $\gamma$ and $c>0$ satisfy
		$\E_\beta|R_n|\geq c\beta^{-\gamma}n^{\frac{d}{d+1}}$ for all $\beta\geq1$ and all
		$n\geq1$, then $\gamma\geq d/(d+1)$. If the real numbers $s$ and $c>0$ satisfy
		$\E_\beta|R_n|\geq c\beta^{-\frac{d}{d+1}}n^{s}$ for all $\beta\geq1$ and all
		$n\geq1$, then $s\leq d/(d+1)$.
	\end{proposition}

	\begin{proof}
		Connectedness gives $|R_n|\leq1+|E_n|$.  After the first step, at least
		one incident edge has been crossed, so the conditional probability of
		crossing an untraversed edge is at most
		$(2d-1)/(\beta+2d-1)\leq2d/\beta$.
		Thus $\E_\beta|E_n|\leq1+2dn/\beta$, proving \eqref{eq:range-upper}.

		For the restrictions on the exponents, take $\beta=n\to\infty$.
		Equation~\eqref{eq:range-upper} bounds $\E_\beta|R_n|$ uniformly in $n$,
		whereas the two hypotheses give lower bounds
		$cn^{d/(d+1)-\gamma}$ and $cn^{s-d/(d+1)}$, respectively.
		In each case, the exponent must be nonpositive.
	\end{proof}

	After its first step, the walk crosses
	the same edge at each of the next $n-1$ steps with probability
	$(\beta/(\beta+2d-1))^{n-1}$, so $\P_\beta\{|R_n|=2\}\to1$ whenever
	$n/\beta\to0$.

	\subsection{Dimension one}\label{ssec:d1}

	In dimension one, the expected range is asymptotic to a multiple of
	$\sqrt n$ as $n\to\infty$ for each fixed $\beta$. That multiple is asymptotic to
	$\sqrt{2/\beta}$ as $\beta\to\infty$.

	\citet[Theorem~1]{PfaffelhuberStiefel2021} computed the asymptotics of every
	moment of the range on $\Z$. In their notation the weight of a traversed
	edge is our $\beta$, and the quantity they call the range
	is $\max_{k\leq n}X_k-\min_{k\leq n}X_k$, which is $|R_n|-1$ here. Their
	first-moment formula gives
	\[
	\lim_{n\to\infty}\frac{\E_\beta|R_n|}{\sqrt n}
	=\sqrt{2/\pi}\int_0^\infty\operatorname{sech}^{2\beta}(x/2)dx
	=\sqrt2\frac{\Gamma(\beta)}{\Gamma(\beta+1/2)}
	\sim\sqrt{2/\beta}
	\qquad(\beta\to\infty)\, .
	\]
	Thus the exponent of $n$ is $\frac12$ for each fixed $\beta$, and the leading
	coefficient has order $\beta^{-1/2}$ as $\beta\to\infty$.
	Both agree with the exponents in \eqref{eq:main-range} at $d=1$.

	\section{Return probabilities}\label{sec:return}

	We bound the expected number of visits to each site by $C\beta n^{1/(d+1)}$
	before time $n$, uniformly over all sites.
	The argument is independent of the range estimates. We use the notation
	$\mu_j$, $E_j$, and $\eta_{j+1}$ from Section~\ref{sec:potential}.

	We evaluate a potential for simple random walk along the reinforced walk.
	Reinforcement contributes $(\beta-1)$ times a sum of potential differences over
	distinct edges. The absolute value of this sum is bounded by the total
	variation of the potential over the edges of $\Z^d$. Truncating the Green
	function of the lazy walk after $R$ steps gives a potential with finite total
	variation. Optimizing over $R$ gives the exponent $1/(d+1)$.

	Let $P$ be the transition operator of simple random walk on $\Z^d$, let
	$Q\coloneqq\frac12(I+P)$ be its lazy version, and for an integer $R\geq1$ let
	\[
		g_R\coloneqq\sum_{r=0}^{R-1}Q^r(0,\cdot)
		\qquad\text{and}\qquad
		\operatorname{osc}(g_R)\coloneqq\max g_R-\min g_R\, .
	\]
	Since $Q-I=\frac12(P-I)$, the sum telescopes:
	\begin{equation}\label{eq:poisson}
		(P-I)(-g_R)=2\bigl(\one_{\{0\}}-Q^R(0,\cdot)\bigr)\, .
	\end{equation}

	\begin{lemma}[Truncated Green function]\label{lem:green}
		Let $d\geq2$ and let $R\geq1$ be an integer.  Then
		\[
			\begin{gathered}
			\|Q^R(0,\cdot)\|_\infty\leq CR^{-d/2},
			\qquad
			\operatorname{osc}(g_R)\leq
			\begin{cases}
				C\log(R+1),&d=2,\\
				C,&d\geq3,
			\end{cases}
			\\
			\sum_{\{x,y\}}|g_R(y)-g_R(x)|\leq4d^{3/2}\sqrt R\, ,
			\end{gathered}
		\]
		the last sum running over the undirected edges of $\Z^d$.
	\end{lemma}

	\begin{proof}
		The Fourier multiplier of $Q$ is
		$\widehat Q(\theta)=\frac12+\frac1{2d}\sum_{i\leq d}\cos\theta_i$, which
		lies in $[0,1]$ and, by $1-\cos t\geq2t^2/\pi^2$ on $[-\pi,\pi]$, is at
		most $\exp(-\|\theta\|_2^2/(d\pi^2))$.  Fourier inversion and a Gaussian
		integral give the first bound, and summing it over $0\leq r<R$, together
		with $g_R\geq0$, gives the second.

		For the third, realize a step of $Q$ by choosing a coordinate uniformly
		and taking a lazy step in it, staying with probability $\frac12$ and
		moving either way with probability $\frac14$.  The $m$-step law of that
		one-dimensional walk is $p_m(k)=4^{-m}\binom{2m}{m+k}$, which is symmetric
		and unimodal, so $\sum_k|p_m(k+1)-p_m(k)|=2p_m(0)$.  Also
		$p_m(0)\leq(m+1)^{-1/2}$, by induction from
		$p_{m+1}(0)/p_m(0)=(2m+1)/(2m+2)$.  Conditionally on how often each
		coordinate is updated, the coordinates are independent, and the number
		$N_i$ of updates of coordinate $i$ in $r$ steps is binomial with
		parameters $r$ and $1/d$. The triangle and Cauchy--Schwarz inequalities,
		together with
		$\E(N_i+1)^{-1}=\bigl(1-(1-1/d)^{r+1}\bigr)d/(r+1)\leq d/(r+1)$,
		give
		\[
		\sum_x\bigl|Q^r(0,x+e_i)-Q^r(0,x)\bigr|
		\leq2\E(N_i+1)^{-1/2}
		\leq2\sqrt{\frac d{r+1}}\, .
		\]
		Summing over the $d$ coordinates and over $0\leq r<R$, and using
		$\sum_{r<R}(r+1)^{-1/2}\leq2\sqrt R$, gives the third bound.
	\end{proof}

	\begin{theorem}[Occupation of a site]\label{thm:return}
		Let $d\geq2$.  There is a constant $C<\infty$, depending only on $d$, such
		that for every $\beta\geq1$, every site $z$, every integer $m\geq0$ and
		every integer $h\geq1$,
		\begin{equation}\label{eq:occupation}
			\E_\beta\sum_{j=m}^{m+h-1}\one_{\{X_j=z\}}
			\leq
			\begin{cases}
				C\bigl[\log(h+2)+(\beta-1)^{2/3}h^{1/3}\bigr],&d=2,\\
				C\bigl[1+(\beta-1)^{\frac{d}{d+1}}h^{\frac{1}{d+1}}\bigr],&d\geq3,
			\end{cases}\, .
		\end{equation}
	\end{theorem}

	\begin{proof}
		First suppose that $m\geq1$. Write $\Delta_jf\coloneqq f(X_{j+1})-f(X_j)$ for a bounded
		$f\colon\Z^d\to\R$. Multiplying $\Delta_jf$ by $\beta$ on first
		crossings makes the coefficient of $f(y)-f(X_j)$ in the conditional
		expectation equal to $\beta/\mu_j$ for every neighbor $y$. Thus
		\begin{equation}\label{eq:compensator}
			\E_\beta\bigl[\Delta_jf+(\beta-1)\eta_{j+1}\Delta_jf\mid\mathcal F_j\bigr]
			=\frac{2d\beta}{\mu_j}(P-I)f(X_j)\, .
		\end{equation}
		Apply \eqref{eq:compensator} to $f\coloneqq-g_R(\cdot-z)$ and let
		$\kappa_j\coloneqq4d\beta/\mu_j$.  By \eqref{eq:poisson} and telescoping,
		\[
		\E_\beta\sum_{j=m}^{m+h-1}\kappa_j\bigl(\one_{\{X_j=z\}}-Q^R(0,X_j-z)\bigr)
		=\E_\beta\Bigl[f(X_{m+h})-f(X_m)
		+(\beta-1)\sum_{j=m}^{m+h-1}\eta_{j+1}\Delta_jf\Bigr]\, .
		\]
		After its first step the walk always stands at a site with a traversed
		incident edge, so $\beta\leq\mu_j\leq2d\beta$ and hence $2\leq \kappa_j\leq4d$ for $j\geq1$.
		The steps with $\eta_{j+1}=1$ cross distinct edges, so the last sum is at
		most the total variation of $g_R$ over all edges of $\Z^d$.  Dividing by
		$2$ and applying Lemma~\ref{lem:green},
		\begin{equation}\label{eq:optimize}
			\begin{aligned}
			\E_\beta\sum_{j=m}^{m+h-1}\one_{\{X_j=z\}}
			&\leq2dh\|Q^R(0,\cdot)\|_\infty+\tfrac12\operatorname{osc}(g_R)\\
			&\qquad+\tfrac{\beta-1}2\sum_{\{x,y\}}|g_R(y)-g_R(x)|\\
			&\leq C\Bigl[hR^{-d/2}+H_d(R)+(\beta-1)\sqrt R\Bigr]\, ,
			\end{aligned}
		\end{equation}
		where $H_2(R)\coloneqq\log(R+1)$ and $H_d(R)\coloneqq1$ for $d\geq3$.  If
		$m=0$, bound the term $j=0$ by $1$ and run the argument from time $1$,
		which adds at most $1$.

		For $0<\beta-1\leq h$, choose $R=\lceil(h/(\beta-1))^{2/(d+1)}\rceil$.
		The terms $hR^{-d/2}$ and $(\beta-1)\sqrt R$ are each at most
		$C(\beta-1)^{\frac{d}{d+1}}h^{1/(d+1)}$.
		This proves \eqref{eq:occupation} for $d\geq3$. For $d=2$, it proves
		\eqref{eq:occupation} when $h^{-1/2}\leq\beta-1\leq h$, since then
		$R\leq h+1$ and $\log(R+1)\leq C\log(h+2)$.  If $d\geq3$ and $\beta=1$, let $R\to\infty$.
		If $d=2$ and $\beta-1<h^{-1/2}$, take $R=h$, which bounds $hR^{-d/2}$ and $(\beta-1)\sqrt R$
		by constants. If $\beta-1>h$, use the trivial
		bound $h$, which is at most
		$(\beta-1)^{\frac{d}{d+1}}h^{1/(d+1)}$ because
		$h^{\frac{d}{d+1}}\leq(\beta-1)^{\frac{d}{d+1}}$.

	\end{proof}

	Taking $m=0$ and $h=n$ in \eqref{eq:occupation}, and using
	$\log(n+2)\leq Cn^{1/3}$ when $d=2$, gives, for every $n\geq1$,
	\begin{equation}\label{eq:cesaro}
		\sup_{z\in\Z^d}\sum_{j=0}^{n-1}\P_\beta\{X_j=z\}
		\leq C\bigl(1+(\beta-1)^{\frac{d}{d+1}}\bigr)n^{\frac{1}{d+1}}\, .
	\end{equation}
	Since $1+(\beta-1)^{\frac{d}{d+1}}\leq2\beta$, this also gives
	\eqref{eq:occupation-intro}.

	\section{Open questions}\label{sec:questions}

\begin{problem}[A matching upper bound]\label{prob:upper}
		Fix $d\geq2$. Is there a threshold $\beta_d$ such that for every $\beta>\beta_d$ there exists $C(\beta,d)<\infty $ with  $\E_\beta|R_n|\leq
		C(\beta,d) \, n^{d/(d+1)}$ for all  $n\geq1$?  Can one take $\beta_2=1$?
	\end{problem}
	
	\begin{problem}[A limit for the rescaled range]\label{prob:limit}
    Suppose that the answer to Problem~\ref{prob:upper} is affirmative for
		$d\geq2$, and fix $\beta>\beta_d$. Does $|R_n|n^{-d/(d+1)}$ converge
		almost surely to a deterministic positive limit? Does
		$n^{-1/(d+1)}R_n$ converge almost surely, in Hausdorff distance, to a
		deterministic compact set? 
	\end{problem}

    A deterministic limit shape under spatial rescaling by $n^{-1/3}$ exists
	for the clockwise rotor walk on $\Z^2$ with independent uniform initial rotors \citep[Theorem~1.1]{BouRabeePeres2026}.
Simulations (see Figure~\ref{fig:range}) suggest that the answer to Problem~\ref{prob:limit} is positive, and the limit shape might be a ball.

	\begin{problem} \label{prob:recurrence}
		Is once-reinforced random walk on $\Z^2$ recurrent for every $\beta>1$?  
	\end{problem}
    Problem \ref{prob:recurrence} is part of the conjecture for $\Z^2$ quoted in
		\citet[Section~1]{CollevecchioTarres2025}.

	\begin{problem}[Universality]\label{prob:universality}
		Which planar walks have expected range of order $n^{2/3}$?
		When their ranges have deterministic limit shapes under spatial
		rescaling by $n^{-1/3}$, are these shapes the same up to dilation?
	\end{problem}

	Both edge and vertex Once-reinforced walks on $\Z^2$ at sufficiently
	large reinforcement are candidates for the same range exponent
	\citep[Section~I]{FosterGrassbergerPaczuski2009}; see also
	Section~\ref{ssec:related-work}.

	\bibliographystyle{plainnat}
	\bibliography{references}

\end{document}